\documentclass[letterpaper, 10 pt, conference]{ieeeconf}  

\IEEEoverridecommandlockouts                              

\usepackage{epsfig} 
\usepackage{amsmath} 
\usepackage{amssymb} 

\newtheorem{theorem}{Theorem}

\usepackage{algorithm,algorithmic}

\title{\LARGE \bf
Sparse preconditioning for model predictive control
}

\author{Andrew Knyazev$^{1}$ and Alexander Malyshev$^{2}$
\thanks{$^{1}$Andrew Knyazev is with Mitsubishi Electric Research Laboratories (MERL)
201 Broadway, 8th floor, Cambridge, MA 02139, USA
        {\tt\small knyazev@merl.com}}%
\thanks{$^{2}$Alexander Malyshev is with the Mitsubishi Electric Research Laboratories (MERL)
201 Broadway, 8th floor, Cambridge, MA 02139, USA
        {\tt\small malyshev@merl.com}}%
}

\begin{document}

\maketitle
\thispagestyle{empty}
\pagestyle{empty}

\begin{abstract}

We propose fast $O(N)$ preconditioning, where $N$ is the number of gridpoints
on the prediction horizon, for iterative solution of (non)-linear systems
appearing in model predictive control methods such as forward-difference
Newton-Krylov methods.
The Continuation/GMRES method for nonlinear model predictive control,
suggested by T.~Ohtsuka in 2004, is a specific application of
the Newton-Krylov method, which uses the GMRES iterative algorithm to solve
a forward difference approximation of the optimality equations on every time step.

\end{abstract}

\section{Introduction}
\label{sec:intro}

The paper deals with novel sparse preconditioning for model predictive control
using, as a specific example, the Continuation/GMRES method for on-line nonlinear model
predictive control suggested by T.~Ohtsuka in~\cite{Oht:04}. The method becomes
popular in solving industrial applications; see, e.g. \cite{HuNaBuKo:15}.
The paper \cite{KnMa:15b} gives guidelines how to use the method in cases,
when the system dynamics obeys a geometric structure, e.g. the symplectic
one, or when the state lies on a smooth manifold. The structure-preserving solver
may increase accuracy of the numerical solution. The paper \cite{KnMa:15c}
treats the problems with the particle solutions for nonlinear MPC using
Continuation/GMRES.

The Continuation/GMRES method is based on Newton-type optimization schemes.
The exact Newton method requires an analytic expression of a corresponding
Jacobian matrix, which is rarely available in practice and is often replaced
with a forward difference (FD) approximation; see, e.g., \cite{Kel:95}.
Such approximate Newton-type optimization schemes utilize the FD approximation
of the original nonlinear equation at every time step.
T.~Ohtsuka uses the GMRES algorithm to solve a finite-difference approximation
$Ax=b$ to the optimality conditions. To cope with possible
ill-conditioning of $A$, the authors of \cite{TaOh:04} propose
a preconditioning strategy, which proved to be not very efficient.

In \cite{KnFuMa:15} and \cite{KnMa:15a}, we systematically search for better preconditioners to
accelerate the GMRES and MINRES convergence in the C/GMRES method. In the present paper,
we propose a sparse efficient $O(N)$ preconditioner for this method,
where $N$ is the number of gridpoints on the prediction horizon.

Another popular approach to numerical solution of MPC problems is developed
in \cite{RaWR:98}, \cite{ShKeCo:10a,ShKeCo:10b}, \cite{WaBo:10} and based
on the interior-point method. The authors of \cite{WaBo:10} develop a direct
method for a linear MPC model with the $O(N)$ arithmetic complexity.
The papers \cite{ShKeCo:10a,ShKeCo:10b} apply the MINRES iteration with
special preconditioners to similar linear MPC problems and prove the $O(N)$
arithmetical complexity of the preconditioned iteration.
In contrast to the above methods, which use the Newton or quasi-Newton approximations,
the recent papers \cite{Gi:13} and \cite{KoFePeDi:15} investigate performance
of the first-order methods and their Nesterov's acceleration.

Our proposed preconditioning technique is essentially based on two ideas.
The symmetric matrix $A$ is a Schur complement of the Hessian of the Lagrangian
function, associated with the model prediction problem. Apart from a few rows
and columns, the preconditioner $M=LU$ is a sort of an incomplete LU factorization
\cite{Ax:94} of the Schur complement without these exceptional rows and columns.
This results in sparse $M$ and its factors, $L$ and $U$, having only $O(N)$
nonzero entries. On the one hand, application of this preconditioner
is almost as fast as that of a diagonal preconditioner.
On the other hand, our preconditioner has high quality leading to fast convergence
of the iterative method, such as GMRES.

The rest of the paper is organized as follows. Section~\ref{sec:predict}
derives the nonlinear equation (\ref{e7}), which solves the model prediction problem,
following \cite{KnFuMa:15,Oht:04}.
We prove the symmetry of the Jacobian matrix for the function defining (\ref{e7})
in Theorem~\ref{th1}. Section~\ref{sec:algo} describes the continuation method for solving
(\ref{e7}). Section~\ref{sec:gmres} formulates the preconditioned GMRES, as in
\cite{KnFuMa:15}. Section~\ref{sec:sprec} describes our new preconditioner.
The preconditioner construction is the main result of the paper.
Section~\ref{sec:ex} illustrates all details of the preconditioner setup
on a representative example. Section~\ref{sec:numer} displays plots of numerical results.

\section{Model prediction problem}
\label{sec:predict}

The model predictive control (MPC) method solves a finite horizon
prediction problem along a fictitious time $\tau\in[t,t+T]$.
Our model finite horizon problem consists, following \cite{KnFuMa:15,Oht:04},
in choosing the control $u(\tau)$
and parameter vector $p$, which minimize the performance index $J$ as follows:
\[
\min_{u,p} J,
\]
where
\[
J = \phi(x(\tau),p)|_{\tau=t+T}+\int_t^{t+T}L(\tau,x(\tau),u(\tau),p)d\tau
\]
subject to the equation for the state dynamics
\begin{equation}\label{e1}
\frac{dx}{d\tau}=f(\tau,x(\tau),u(\tau),p),
\end{equation}
and the constraints
\begin{equation}\label{e2}
C(\tau,x(\tau),u(\tau),p) = 0,
\end{equation}
\begin{equation}\label{e3}
\psi(x(\tau),p)|_{\tau=t+T} = 0.
\end{equation}
The initial value condition $x(\tau)|_{\tau=t}$ for (\ref{e1}) is the state vector $x(t)$
of the dynamic system. The control vector $u=u(\tau)$, solving the problem over the horizon,
is used as an input to control the system at time $t$. The components of the vector $p(t)$
are parameters of the system. Equation (\ref{e1}) describes the system dynamic
that may be nonlinear in $x$ and $u$. Equations (\ref{e2}) and (\ref{e3}) give equality
constraints for the state $x$ and the control $u$. The horizon time length $T$
may in principle also depend on $t$.

The continuous formulation of the finite horizon problem stated above is discretized
on a uniform, for simplicity of presentation, time grid over the horizon $[t,t+T]$
partitioned into $N$ time
steps of size $\Delta\tau$, and the time-continuous vector functions $x(\tau)$ and $u(\tau)$ are
replaced by their indexed values $x_i$ and $u_i$ at the grid points $\tau_i$,
$i=0,1,\ldots,N$. The integral of the performance cost $J$ over the horizon is approximated by the rectangular quadrature rule. The time derivative of the state vector is approximated by the forward difference formula. The discretized optimal control problem is formulated as follows:
\[
\min_{u_i,p}\left[
\phi(x_N,p) + \sum_{i=0}^{N-1}L(\tau_i,x_i,u_i,p)\Delta\tau\right],
\]
subject to
\begin{equation}\label{e4}
\quad x_{i+1} = x_i + f(\tau_i,x_i,u_i,p)\Delta\tau,\quad i = 0,1,\ldots,N-1,
\end{equation}
\begin{equation}\label{e5}
C(\tau_i,x_i,u_i,p) = 0,\quad  i = 0,1,\ldots,N-1,
\end{equation}
\begin{equation}\label{e6}
\psi(x_N,p) = 0.
\end{equation}

The necessary optimality conditions for the discretized finite horizon problem
are obtained by means of the discrete Lagrangian function
\begin{eqnarray*}
&&\mathcal{L}(X,U)=\phi(x_N,p)+\sum_{i=0}^{N-1}
L(\tau_i,x_i,u_i,p)\Delta\tau\\
&&+\,\lambda_0^T[x(t)-x_0]\\
&&+\sum_{i=0}^{N-1}\lambda_{i+1}^T[x_i-x_{i+1}+f(\tau_i,x_i,u_i,p)\Delta\tau]\\
&&+\sum_{i=0}^{N-1}\mu_i^TC(\tau_i,x_i,u_i,p)\Delta\tau+\nu^T\psi(x_N,p),
\end{eqnarray*}
where $X = [x_i\; \lambda_i]^T$, $i=0,1,\ldots,N$, and
$U = [u_i\; \mu_i\; \nu\; p]^T$, $i=0,1,\ldots,N-1$.
Here, $\lambda$ is the costate vector, $\mu$ is the Lagrange multiplier vector 
associated with the constraint~(\ref{e5}). The terminal constraint (\ref{e6})
is relaxed by the aid of the Lagrange multiplier $\nu$. For further covenience,
we also introduce the Hamiltonian function
\begin{eqnarray*}
\lefteqn{H(t,x,\lambda,u,\mu,p) = L(t,x,u,p)}\hspace*{6em}\\
&&{}+\lambda^Tf(t,x,u,p)+\mu^TC(t,x,u,p).
\end{eqnarray*}

The necessary optimality conditions are the (KKT) stationarity conditions:
$\mathcal{L}_{\lambda_i}=0$, $\mathcal{L}_{x_i}=0$, $i=0,1,\ldots,N$,
$\mathcal{L}_{u_j}=0$, $\mathcal{L}_{\mu_j}=0$, $i=0,1,\ldots,N-1$,
$\mathcal{L}_{\nu_k}=0$, $\mathcal{L}_{p_l}=0$.

The KKT conditions are reformulated in terms of a mapping $F[U,x,t]$,
where the vector $U$ combines the control input $u$, the Lagrange multiplier
$\mu$, the Lagrange multiplier $\nu$, and the parameter $p$, all in one vector:
\[
U(t)=[u_0^T,\ldots,u_{N-1}^T,\mu_0^T,\ldots,\mu_{N-1}^T,\nu^T,p^T]^T. 
\]
The vector argument $x$ in $F[U,x,t]$ denotes the current measured or estimated 
state vector, which serves as the initial vector $x_0$ in the following procedure.
\begin{enumerate}
\item Starting from the current measured or estimated state $x_0$, compute
$x_i$, $i=0,1\ldots,N-1$, by the forward recursion
\[
x_{i+1} = x_i + f(\tau_i,x_i,u_i,p)\Delta\tau.
\]
Then starting from
\[
\lambda_N=\frac{\partial\phi^T}{\partial x}(x_N,p)+
 \frac{\partial\psi^T}{\partial x}(x_N,p)\nu
\]
compute the costates $\lambda_i$, $i=N,N\!-\!1,\ldots,1$, by the backward recursion
\[
\lambda_i=\lambda_{i+1}+\frac{\partial H^T}{\partial x}
(\tau_i,x_i,\lambda_{i+1},u_i,\mu_i,p)\Delta\tau.
\]
\item Calculate $F[U,x,t]$, using just obtained $x_i$ and $\lambda_i$, as
\begin{eqnarray*}
\lefteqn{F[U,x,t]}\\
&&\hspace*{-2em}=\left[\begin{array}{c}\begin{array}{c}
\frac{\partial H^T}{\partial u}(\tau_0,x_0,\lambda_{1},u_0,\mu_0,p)\Delta\tau\\
\vdots\\\frac{\partial H^T}{\partial u}(\tau_i,x_i,\lambda_{i+1},u_i,\mu_i,p)\Delta\tau\\
\vdots\\\frac{\partial H^T}{\partial u}(\tau_{N-1},x_{N-1},\lambda_{N},u_{N-1},
\mu_{N-1},p)\Delta\tau\end{array}\\\;\\
\begin{array}{c}C(\tau_0,x_0,u_0,p)\Delta\tau\\
\vdots\\C(\tau_i,x_i,u_i,p)\Delta\tau\\\vdots\\
C(\tau_{N-1},x_{N-1},u_{N-1},p)\Delta\tau\end{array}\\\;\\
\psi(x_N,p)\\[2ex]
\begin{array}{c}\frac{\partial\phi^T}{\partial p}(x_N,p)+
\frac{\partial\psi^T}{\partial p}(x_N,p)\nu\\
+\sum_{i=0}^{N-1}\frac{\partial H^T}{\partial p}(\tau_i,x_i,
\lambda_{i+1},u_i,\mu_i,p)\Delta\tau\end{array}
\end{array}\right].
\end{eqnarray*}
\end{enumerate}

The equation with respect to the unknown vector $U(t)$
\begin{equation}\label{e7}
 F[U(t),x(t),t]=0
\end{equation}
gives the required necessary optimality conditions.

\begin{theorem}\label{th1}
The Jacobian matrix $F_U[U,x,t]$ is symmetric for all $U$, $x$, and $t$.
\end{theorem}
\begin{proof}
The equation $\mathcal{L}_X(X,U)=0$ is always solvable with respect to $X$
by the forward recursion for $x_i$ and backward recursion for $\lambda_i$.
Let us denote its solution by $X=g(U)$. Then $F[U]=\mathcal{L}_U(g(U),U)$ and
$F_U=\mathcal{L}_{UU}(g(U),U)+\mathcal{L}_{UX}(g(U),U)g_U$.
Differentiation of the identity $\mathcal{L}_U(g(U),U)=0$ with respect to $U$
gives the identity $\mathcal{L}_{UU}(g(U),U)+\mathcal{L}_{UX}(g(U),U)g_U(U)=0$.
Differentiation of the identity $\mathcal{L}_X(g(U),U)=0$ with respect to $U$
gives the identity $\mathcal{L}_{XU}(g(U),U)+\mathcal{L}_{XX}(g(U),U)g_U(U)=0$.
Hence $g_U=-\mathcal{L}_{XX}^{-1}(g(U),U)\mathcal{L}_{XU}(g(U),U)$ and
\begin{align}
F_U[U] =& \mathcal{L}_{UU}(g(U),U)\label{e7a}\\
&{}-\mathcal{L}_{UX}(g(U),U)\mathcal{L}_{XX}^{-1}(g(U),U)\mathcal{L}_{XU}(g(U),U),\notag
\end{align}
which is the Schur complement of the symmetric Hessian matrix of $\mathcal{L}$
at the point $(X,U)=(g(U),U)$. The Schur complement of any symmetric matrix is
symmetric.
\end{proof}

\section{Continuation algorithm}
\label{sec:algo}

The controlled system is sampled on a uniform time grid $t_j=j\Delta t$, $j=0,1,\ldots$.
Solution of equation (\ref{e7}) must be found at each time step $t_j$ on the controller board,
which is a challenging part of implementation of NMPC.

Let us denote $x_j=x(t_j)$, $U_j=U(t_j)$, and rewrite the equation $F[U_j,x_j,t_j]=0$
equivalently in the form
\[
F[U_j,x_j,t]-F[U_{j-1},x_j,t_j]=b_j,
\]
where
\begin{equation}\label{e8}
b_j=-F[U_{j-1},x_j,t_j].
\end{equation}

Using a small $h$, which may be different from $\Delta t$ and $\Delta\tau$,
we introduce the forward difference operator
\begin{eqnarray}\label{e9}
a_j(V)=(F[U_{j-1}+hV,x_j,t_j]-F[U_{j-1},x_j,t_j])/h.
\end{eqnarray}
We note that the equation $F[U_j,x_j,t_j]=0$ is equivalent to the equation
$a_j(\Delta U_j/h)=b_j/h$, where $\Delta U_j=U_j-U_{j-1}$.

Let us denote the $k$-th column of the $m\times m$ identity matrix by $e_k$,
where $m$ is the dimension of the vector $U$, and define an $m\times m$ matrix $A_j$
with the columns $A_je_k$, $k=1,\ldots,m$, given by the formula $A_je_k=a_j(e_k)$.
The matrix $A_j$ is an $O(h)$ approximation of the Jacobian matrix
$F_U[U_{j-1},x_j,t_j]$. The Jacobian matrix $F_U$ is symmetric by Theorem~\ref{th1}.

Suppose that an approximate solution $U_0$ to the equation $F[U_0,x_0,t_0]=0$ is
available. The first block entry of $U_0$ is then taken as the control $u_0$
at the state $x_0$. The next state $x_1=x(t_1)$ is either sensor estimated or computed
by the formula $x_1=x_0+f(t_0,x_0,u_0)\Delta t$; cf. (\ref{e1}).

At the time $t_j$, $j>1$, we have the state $x_j$ and the vector $U_{j-1}$
from the previous time $t_{j-1}$. Our goal is to solve the following equation
with respect to $V$:
\begin{equation}\label{e11}
 a_j(V)=b_j/h.
\end{equation}
Then we set $\Delta U_j=hV$, $U_j=U_{j-1}+\Delta U_j$ and choose the first block
component of $U_j$ as the control $u_j$. The next system state $x_{j+1}=x(t_{j+1})$
is either sensor estimated or computed by the formula $x_{j+1}=x_j+f(t_j,x_j,u_j)\Delta t$.

A direct way to solve (\ref{e11}) is generating the matrix $A_j$
and then solving the system of linear equations $A_j\Delta U_j=b_j$;
e.g., by the Gaussian elimination.

A less expensive alternative is solving (\ref{e11}) by the GMRES method,
where the operator $a_j(V)$ is used without explicit construction of
the matrix $A_j$; cf., \cite{Kel:95,Oht:04}.

\section{Preconditioned GMRES}
\label{sec:gmres}

We recall that, for a given system of linear equations $Ax=b$ and initial approximation $x_0$,
GMRES constructs orthonormal bases of the Krylov subspaces
$\mathcal{K}_n=\text{span}\{r_0,Ar_0,\ldots,A^{n-1}r_0\}$,
$n=1,2,\ldots$, given by the columns of matrices $Q_n$,
such that $AQ_n=Q_{n+1}H_{n}$ with the upper Hessenberg matrices~$H_n$
and then searches for approximations to the solution $x$ in the form $x_n=Q_ny_n$,
where $y_n=\text{argmin}\|AQ_ny_n-b\|_2$.

The convergence of GMRES may stagnate for an ill-conditioned matrix $A$.
The convergence can be improved by preconditioning. A matrix $M$ that is close
to the matrix $A$ and such that computing $M^{-1}r$ for an arbitrary vector $r$
is relatively easy, is referred to as a preconditioner. The preconditioning
for the system of linear equations $Ax=b$ with the preconditioner $M$ formally
replaces the original system $Ax=b$ with the equivalent preconditioned linear system
$M^{-1}Ax=M^{-1}b$. If the condition number $\|M^{-1}A\|\|A^{-1}M\|$
of the matrix $M^{-1}A$ is small, convergence of
iterative solvers for the preconditioned system can be fast.

A typical implementation of the preconditioned GMRES is given by Algorithm \ref{a1}.
GMRES without preconditioning is the same algorithm with $z=r$. In the pseudocode,
we denote by $H_{i_1:i_2,j_1:j_2}$ the submatrix of $H$ with the entries $H_{ij}$
such that $i_1\leq i\leq i_2$ and $j_1\leq j\leq j_2$.

\begin{algorithm}
\caption{Preconditioned GMRES($k_{\max}$)}
\begin{algorithmic}[1]\label{a1}
\REQUIRE $a(v)$, $b$, $x_0$, $k_{\max}$, $M$
\ENSURE Solution $x$ of $a(x)=b$
\STATE $r=b-a(x_0)$, $z=M^{-1}r$, $\beta=\|z\|_2$, $v_1=z/\beta$
\FOR{$k=1,\ldots,k_{\max}$}
\STATE $r=a(v_k)$, $z=M^{-1}r$
\STATE $H_{1:k,k}=[v_1,\ldots,v_k]^Tz$
\STATE $z=z-[v_1,\ldots,v_k]H_{1:k,k}$
\STATE $H_{k+1,k}=\|z\|_2$
\STATE $v_{k+1}=z/\|z\|_2$
\ENDFOR
\STATE $y=\mbox{arg min}_y\|H_{1:k_{\max}+1,1:k_{\max}}y-[\beta,0,\dots,0]^T\|_2$
\STATE $x=x_0+[v_1,\ldots,v_{k_{\max}}]y$
\end{algorithmic}
\end{algorithm}

It is a common practice to compute the LU factorization $M=LU$ by the Gaussian
elimination and then compute the vector $M^{-1}r$ by the rule
$M^{-1}r=U^{-1}(L^{-1}r)$.

\section{Sparse preconditioner}
\label{sec:sprec}

Our finite horizon model prediction problem allows us to construct sparse preconditioners
$M_j$ with a particular structure. These preconditioners are highly efficient,
which is confirmed by the numerical experiments described below.

We first observe that the states $x_i$, computed by the forward recursion, and
the costates $\lambda_i$, computed by the subsequent backward recursion, satisfy,
in practice, the following property: $\partial x_{i_1}/\partial u_{i_2}=O(\Delta\tau)$,
$\partial\lambda_{i_1}/\partial u_{i_2}=O(\Delta\tau)$,
$\partial x_{i_1}/\partial\mu_{i_2}=0$ and
$\partial\lambda_{i_1}/\partial\mu_{i_2}=O(\Delta\tau)$.
It is a corollary of theorems about the derivatives of solutions of ordinary
differential equations with respect to a parameter; see, e.g.,~\cite{Pon:62}.

Now we assume that the predicted states $x_i$ and costates $\lambda_i$ are computed
by the forward and backward recursions for the vector $U_{j-1}$ at the
current system state $x_j=x(t_j)$ during computation of the right-hand side vector
$b_j$ and use the predicted $x_i$ and $\lambda_i$ to form the blocks
$H_{uu}$, $H_{u\mu}$, $H_{\mu u}$, $H_{\mu\mu}$ of the symmetric matrix
\[
M_j = \left[\begin{array}{ccc}H_{uu}(U_{j-1},x_j,t_j)&H_{u\mu}(U_{j-1},x_j,t_j)&
M_{13}\\H_{\mu u}(U_{j-1},x_j,t_j)&H_{\mu\mu}(U_{j-1},x_j,t_j)&M_{23}\\
M_{31}&M_{32}&M_{33}\end{array}\right],
\]
where $[M_{31},M_{32},M_{33}]$ coincides with the last $l$ rows of
$A_j$. The integer $l$ denotes the sum of dimensions of $\psi$ and $p$.

In the notation of Theorem~\ref{th1}, the above construction is explained as follows.
We discard the second term in formula~\ref{e7a} and use
the truncated expression $F_U[U]=\mathcal{L}_{UU}(g(U),U)$
for the entries of $M_j$ apart from the last $l$ columns and last $l$ rows.
The last $l$ columns are computed exactly, the last $l$ rows equal the transposed
last $l$ columns because of the symmetry of $M_j$. The possibility to
use the truncated expression is due to the above observation that
$\partial x_{i_1}/\partial u_{i_2}=O(\Delta\tau)$,
$\partial\lambda_{i_1}/\partial u_{i_2}=O(\Delta\tau)$,
$\partial x_{i_1}/\partial\mu_{i_2}=0$,
$\partial\lambda_{i_1}/\partial\mu_{i_2}=O(\Delta\tau)$.
Moreover, the norm of
$A_j-M_j$ is of order $O(\Delta\tau)$.

The matrix $M_j$ is sparse since the blocks
$H_{uu}$, $H_{u\mu}$, $H_{\mu u}$, $H_{\mu\mu}$
are block diagonal and $l$ is small.
The particular structure of $M_j$ is convenient for efficient LU factorization.
It is possible to simultaneously permute the rows and columns of $M_j$
and to obtain an arrow-like pattern of nonzero elements, which admits
a fast LU factorization. A representative example of the sparse preconditioners
$M_j$ and their LU factorization is given in the next section. 

As a result, the setup of $M_j$, computation of its LU factorization, and application
of the preconditioner all cost $O(N)$ floating point operations.
The memory requirements are also of order $O(N)$.

\section{Example}
\label{sec:ex}

We consider a test nonlinear problem, which describes the minimum-time motion from a state
$(x_0,y_0)$ to a state $(x_f,y_f)$ with an inequality constrained control:
\begin{itemize}
\item State vector
 $\vec{x}=\left[\begin{array}{c}x\\y\end{array}\right]$ and
input control $\vec{u}=\left[\begin{array}{c}u\\u_d\end{array}\right]$.
\item Parameter variables $\vec{p}=[t_f]$, where $t_f$ denotes the length
of the evaluation horizon.
\item Nonlinear dynamics is governed by the system of ODE
\[
\dot{\vec{x}}=f(\vec{x},\vec{u},\vec{p})=
\left[\begin{array}{c}(Ax+B)\cos u\\(Ax+B)\sin u\end{array}\right].
\]
\item Constraints: $C(\vec{x},\vec{u},\vec{p})=[(u-c_{u})^2+u_d^2-r_{u}^2]=0$,
where $c_u=c_0+c_1\sin(\omega t)$, i.e., the control $u$ always stays within
the curvilinear band $c_{u}-r_{u}\leq u\leq c_{u}+r_{u}$).
\item Terminal constraints: $\psi(\vec{x},\vec{p})=\left[\begin{array}{c}
x-x_f\\y-y_f\end{array}\right]=0$ (the state should pass through the point
$(x_f,y_f)$ at $t=t_f$)
\item Objective function to minimize:
\[
J=\phi(\vec{x},\vec{p})+\int_t^{t+t_f}L(\vec{x},\vec{u},\vec{p})dt,
\]
where
\[
\phi(\vec{x},\vec{p})=t_f,\quad L(\vec{x},\vec{u},\vec{p})=-w_{d}u_d
\]
(the state should arrive at $(x_f,y_f)$ in the shortest time;
the function $L$ serves to stabilize the slack variable $u_d$)
\item Constants: $A=B=1$, $x_0=y_0=0$, $x_f=y_f=1$,
$c_0=0.8$, $c_1=0.3$, $\omega=20$, $r_{u}=0.2$, $w_{d}=0.005$.
\end{itemize}
The horizon $[t,t+t_f]$ is parameterized by the affine mapping
$\tau\to t+\tau t_f$ with $\tau\in[0,1]$.

The components of the corresponding discretized problem on the horizon are given below:
\begin{itemize}
\item $\Delta\tau=1/N$, $\tau_i=i\Delta\tau$,
$c_{ui}=c_0+c_1\sin(\omega(t+\tau_ip))$;
\item the participating variables are the state $\left[\begin{array}{c}
x_i\\y_i\end{array}\right]$, the costate $\left[\begin{array}{c}
\lambda_{1,i}\\\lambda_{2,i}\end{array}\right]$, the control $\left[\begin{array}{c}
u_{i}\\u_{di}\end{array}\right]$, the Lagrange multipliers
$\mu_i$ and $\left[\begin{array}{c}\nu_{1}\\\nu_{2}\end{array}\right]$,
the parameter $p$;
\item the state is governed by the model equation
\[
\left\{\begin{array}{l} x_{i+1}=x_i+\Delta\tau\left[p\left(Ax_{i}+B\right)\cos u_{i}\right],\\
y_{i+1}=y_i+\Delta\tau\left[p\left(Ax_{i}+B\right)\sin u_{i}\right],\end{array}\right.
\]
where $i=0,1,\ldots,N-1$;
\item the costate is determined by the backward recursion ($\lambda_{1,N}=\nu_1$,
$\lambda_{2,N}=\nu_2$)
\[
\left\{\begin{array}{l} \lambda_{1,i}=\lambda_{1,i+1}\\
\hspace{2.5em}{}+\Delta\tau\left[pA(\cos u_i \lambda_{1,i+1}+\sin u_i\lambda_{2,i+1})\right],\\
\lambda_{2,i} = \lambda_{2,i+1},\end{array}\right.
\]
where $i=N-1,N-2,\ldots,0$;
\item the equation $F(U,x_0,y_0,t)=0$, where
\begin{eqnarray*}
\lefteqn{U=[u_0,u_{d,0},\ldots,u_{N-1},u_{d,N-1},}\hspace*{8em}\\
&&\mu_0,\ldots,\mu_{N-1},\nu_1,\nu_2,p],
\end{eqnarray*}
has the following rows from the top to bottom:
\[
\left\{\begin{array}{l}
\Delta\tau\left[p(Ax_i+B)\left(-\sin u_i\lambda_{1,i+1}+
\cos u_i\lambda_{2,i+1}\right)\right.\\
\hspace*{11em}\left.{}+2\left(u_i-c_{ui}\right)\mu_i\right] = 0 \\
\Delta\tau\left[2\mu_iu_{di}-w_{d}p\right] = 0 \end{array}\right.
\]
\[
\left\{\Delta\tau\left[(u_i-c_{ui})^{2}+u_{di}^2-r_{u}^2\right]=0
\right.\hspace*{7em}
\]
\[
\left\{\begin{array}{l} x_N-x_r=0\\y_N-y_r=0\end{array}\right.\hspace{15em}
\]
\[
\left\{\begin{array}{l}\Delta\tau[\sum\limits^{N-1}_{i=0}
(Ax_i+B)(\cos u_i\lambda_{1,i+1}+\sin u_i\lambda_{2,i+1})\\
\hspace{1em}{}-2(u_i-c_{ui})\mu_ic_1\cos(\omega(t+\tau_ip))\omega\tau_i\\
\hspace{8em}-w_du_{di}]+1 = 0.\end{array}\right.
\]
\end{itemize}

The matrices $A_j$ have the sparsity structure as in Fig.~\ref{fig4}. 
The preconditioner $M_j$ is the symmetric matrix
\[
M_j = \left[\begin{array}{ccccc}M_{11}&0&M_{13}&M_{14}&M_{15}\\
0&M_{22}&M_{23}&0&M_{25}\\M_{31}&M_{32}&0&0&M_{35}\\
M_{41}&0&0&0&M_{45}\\M_{51}&M_{52}&M_{53}&M_{45}&M_{55}\end{array}\right]
\]
having the diagonal blocks $M_{11}$, $M_{13}=M_{31}^T$, $M_{22}$,
$M_{23}=M_{32}^T$. The diagonal entries of $M_{11}$ equal
$\Delta\tau[2\mu_i-p(Ax_i+B)(\cos u_i\lambda_{1,i+1}+\sin u_i\lambda_{2,i+1})]$.
The diagonal entries of $M_{22}$ equal $\Delta\tau2\mu_i$.
The diagonal entries of $M_{13}$ equal $\Delta\tau2(u_i-c_{ui})$.
The diagonal entries of $M_{23}$ equal $\Delta\tau2u_{di}$.
The entries of the vector $M_{15}$ equal $\Delta\tau(Ax_i+B)
(-\sin u_i\lambda_{1,i+1}+\cos u_i\lambda_{2,i+1})-\Delta\tau2\mu_i
c_1\cos(\omega(t+\tau_ip))\omega\tau_i$. The entries of the vector
$M_{25}$ equal $-\Delta\tau w_d$. The entries of the vector $M_{35}$
equal $-2\Delta\tau(u_i-c_{ui})c_1\cos(\omega(t+\tau_i p))\omega\tau_i$.

The blocks $M_{14}$, $M_{45}$, and $M_{55}$ equal to the respective
blocks of $A$ and have to be computed exactly. The sparsity pattern of $M_j$
is displayed in Fig.~\ref{fig4}.

To compute the LU factorization of $M_j$ with $O(N)$ floating point operations, we
first repartition $M_j$ as
\[
M_j=\left[\begin{array}{ccccc}K_{11}&K_{12}\\K_{21}&K_{22}\end{array}\right],
K_{11}=\left[\begin{array}{ccccc}M_{11}&0&M_{13}\\
0&M_{22}&M_{23}\\M_{31}&M_{32}&0\end{array}\right],
\]
where $K_{11}$ is usually nonsingular. Using the representation
\[
K_{11}^{-1}=\left[\begin{array}{rrr}M_{23}M_{32}&-M_{13}M_{32}&M_{13}M_{22}\\
-M_{23}M_{31}&M_{13}M_{31}&M_{11}M_{23}\\M_{22}M_{31}&M_{11}M_{32}&-M_{11}M_{22}
\end{array}\right]\\
\]
\[
\times\left[\begin{array}{rrr}D\\&D\\&&D\end{array}\right],
\]
where $D=(M_{11}M_{23}M_{32}+M_{13}M_{22}M_{31})^{-1}$, we obtain
the block triangular factors
\[
L = \left[\begin{array}{cc}I&0\\K_{21}K^{-1}_{11}&I\end{array}\right],\;
U = \left[\begin{array}{cc}K_{11}&K_{12}\\0&S_{22}\end{array}\right],
\]
where $S_{22}=K_{22}-K_{21}K_{11}^{-1}K_{12}$.
The application of the preconditioner costs $O(N)$ operations.

An alternative construction of the LU factorization uses a suitable simultaneous
permutation of the rows and columns of $M_j$ with the permutation indices
$1,1+N,1+2N$,\ldots, $i,i+N,i+2N$,\ldots,$N,2N,3N,1+3N,2+3N,3+3N$.
The sparsity patterns of the permuted matrix and its lower triangular factor
$L$ are displayed in Fig.~\ref{fig5}, the sparsity pattern of the upper
triangular factor $U$ is the transpose of that of the factor $L$.

\section{Numerical results}
\label{sec:numer}

In our numerical experiments, carried out in MATLAB, the system of weakly nonlinear
equations (\ref{e11}) for the test problem from Section~\ref{sec:ex} is solved
by the GMRES method. The error tolerance in GMRES is $tol=10^{-5}$.
The number of grid points on the horizon is $N=100$, the sampling time
of the simulation is $\Delta t=1/500$, and $h=10^{-8}$. 

The sparse preconditioners for GMRES are constructed as in Section~\ref{sec:ex}, 
and the LU factorization is computed as proposed in the last paragraph
of Section~\ref{sec:ex}.

Figure \ref{fig1} shows the computed trajectory for the test example and
Figure \ref{fig2} shows the optimal control by the MPC approach using
GMRES with preconditioning.

GMRES with preconditioning executes only 2 iterations at each step
while keeping $\|F\|_2$ close to $10^{-4}$.
For comparison, we show the number of iterations in GMRES without preconditioning
in Figure \ref{fig3}, which is 4-14 times larger. 

\begin{figure}[ht]
\noindent\centering{
\includegraphics[width=\linewidth]{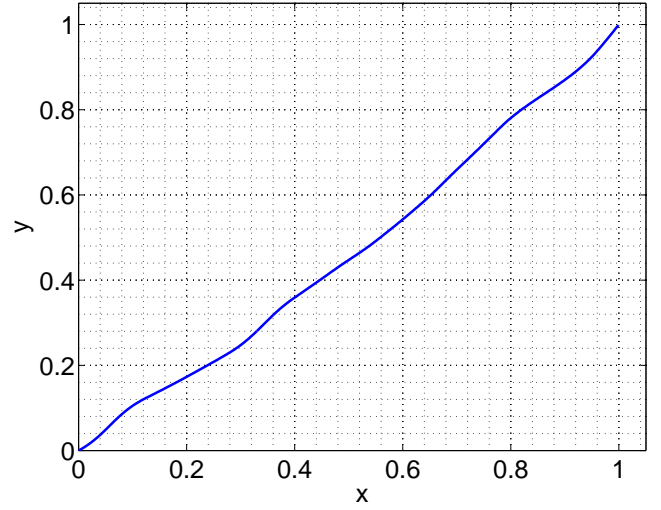}
}
\caption{Trajectory by NPMC using GMRES with preconditioning}
\label{fig1}
\end{figure}
 
\begin{figure}[ht]
\noindent\centering{
\includegraphics[width=\linewidth]{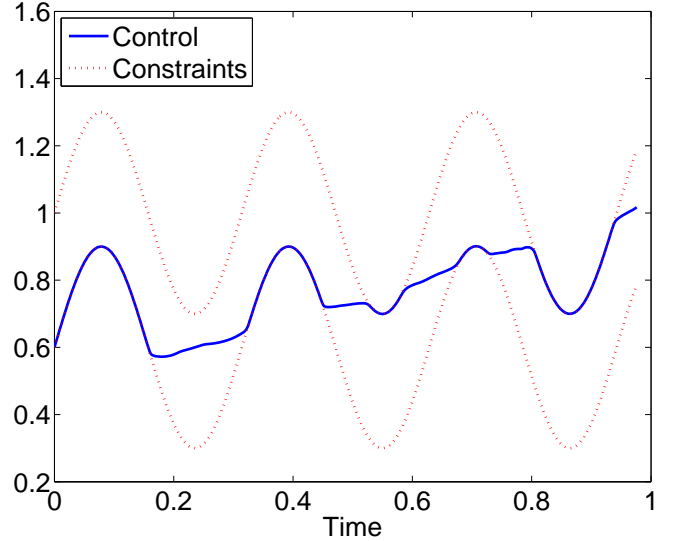}
}
\caption{NMPC control $u$ using GMRES with preconditioning}
\label{fig2}
\end{figure}

\begin{figure}[ht]
\noindent\centering{
\includegraphics[width=\linewidth]{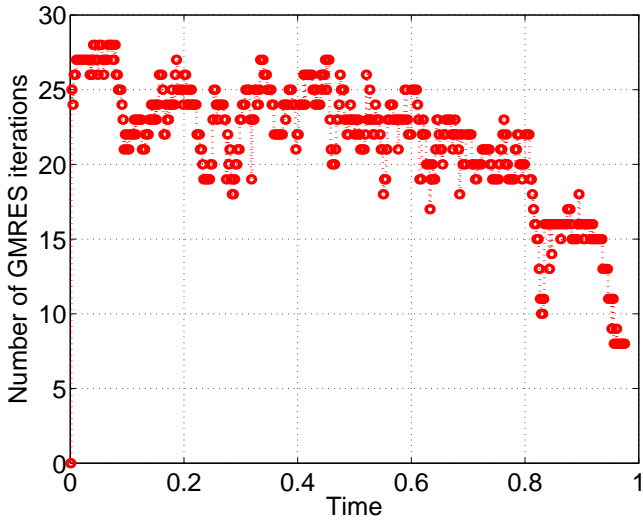}
}
\caption{The number of GMRES iterations without preconditioning,
$N=100$, $\Delta t=1/500$, $k_{\max}=100$}
\label{fig3}
\end{figure}

\begin{figure}[ht]
\noindent\centering{
\includegraphics[width=12em]{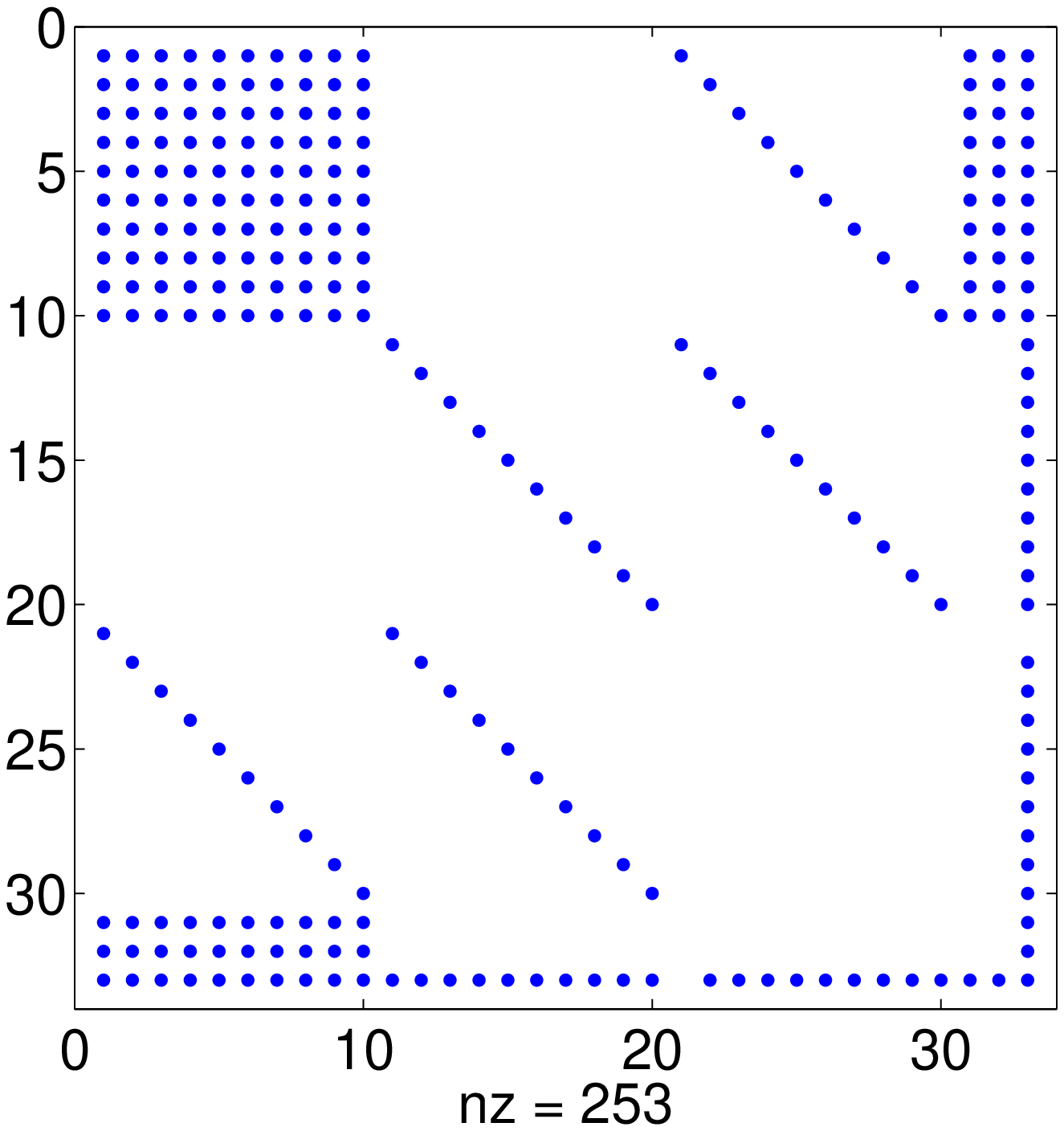}\hfill
\includegraphics[width=12em]{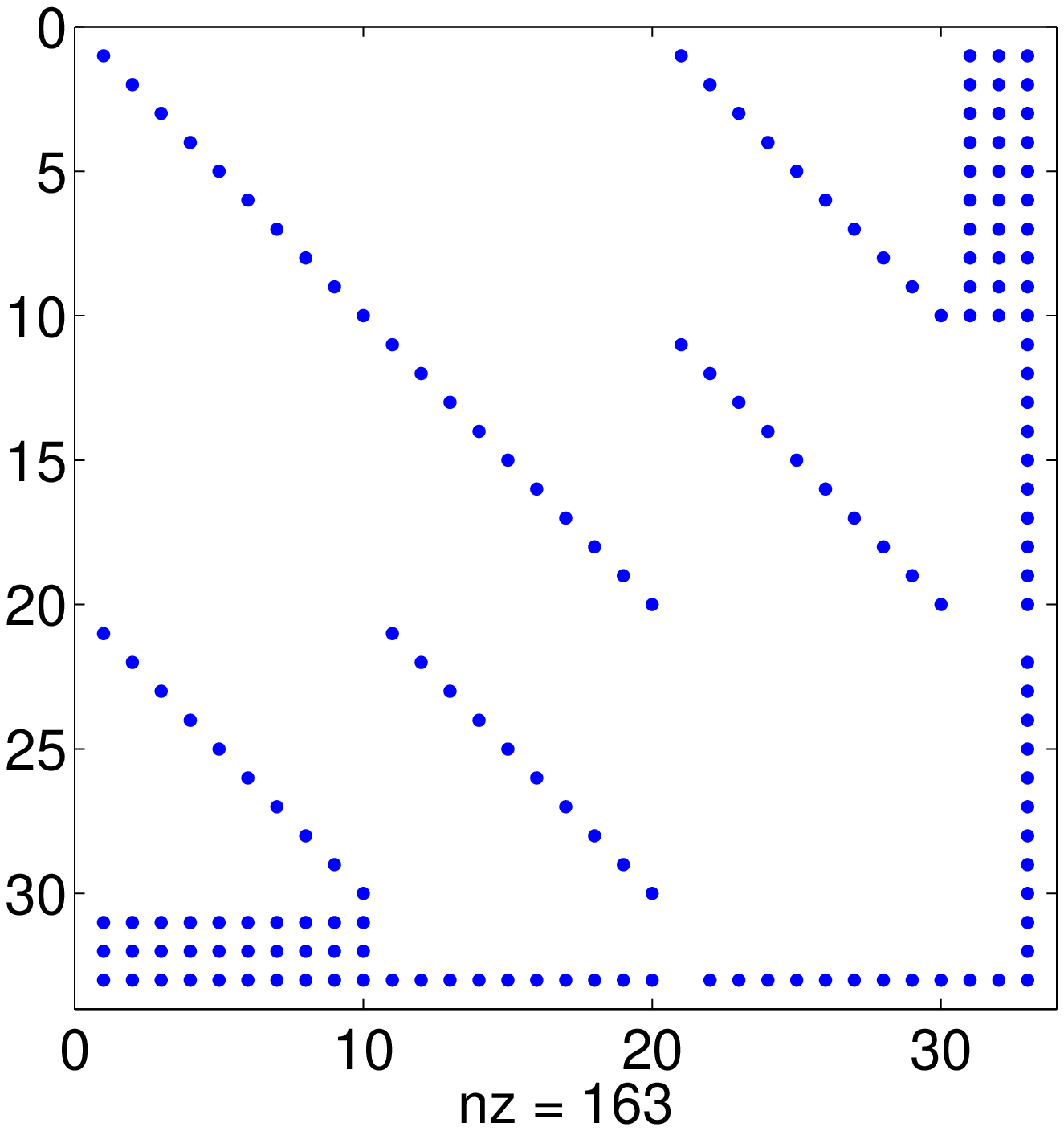}
}
\caption{Sparsity patterns of the Jacobian $F_U$ and preconditioner $M$}
\label{fig4}
\end{figure}

\begin{figure}[ht]
\noindent\centering{
\includegraphics[width=12em]{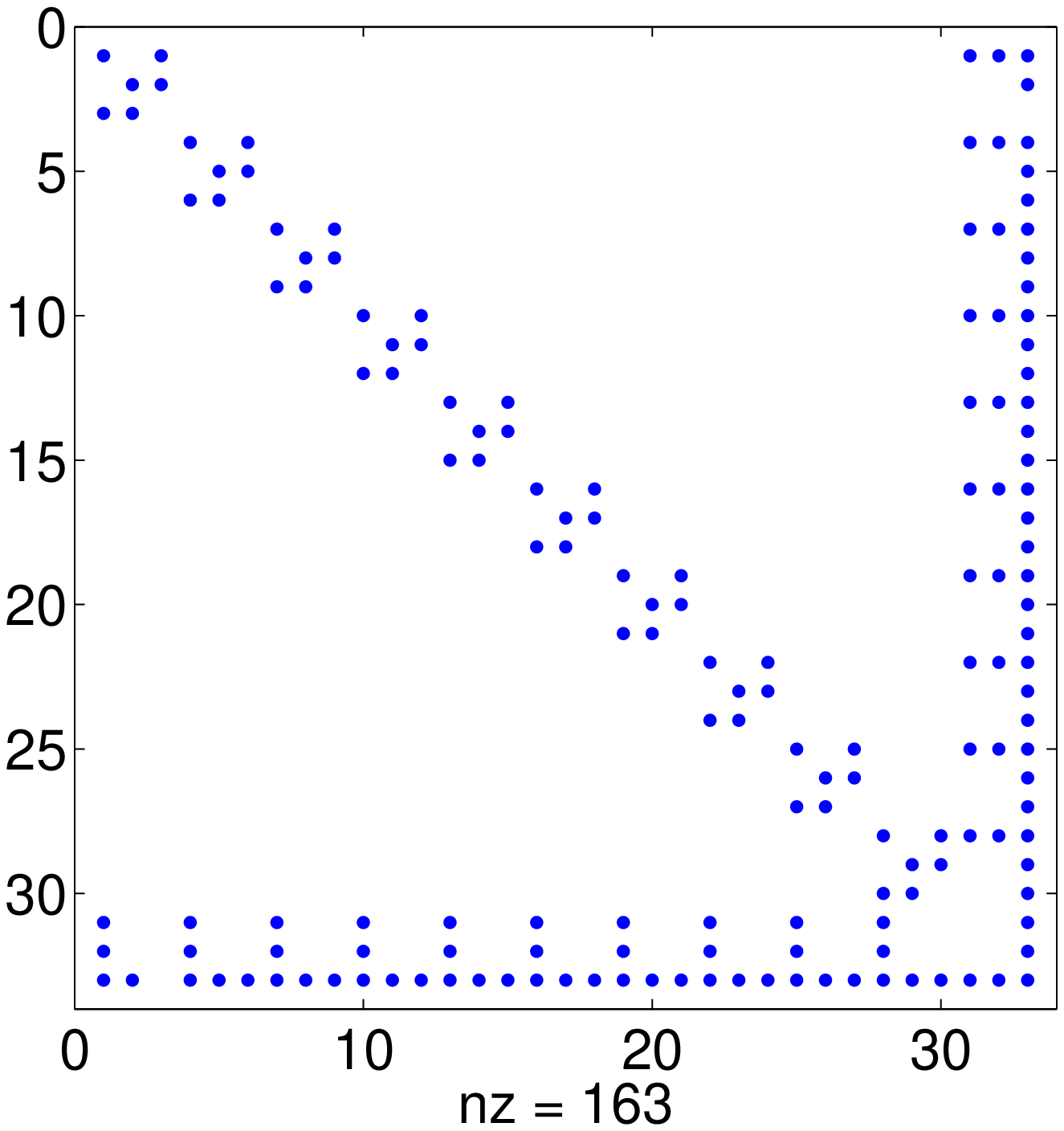}\hfill
\includegraphics[width=12em]{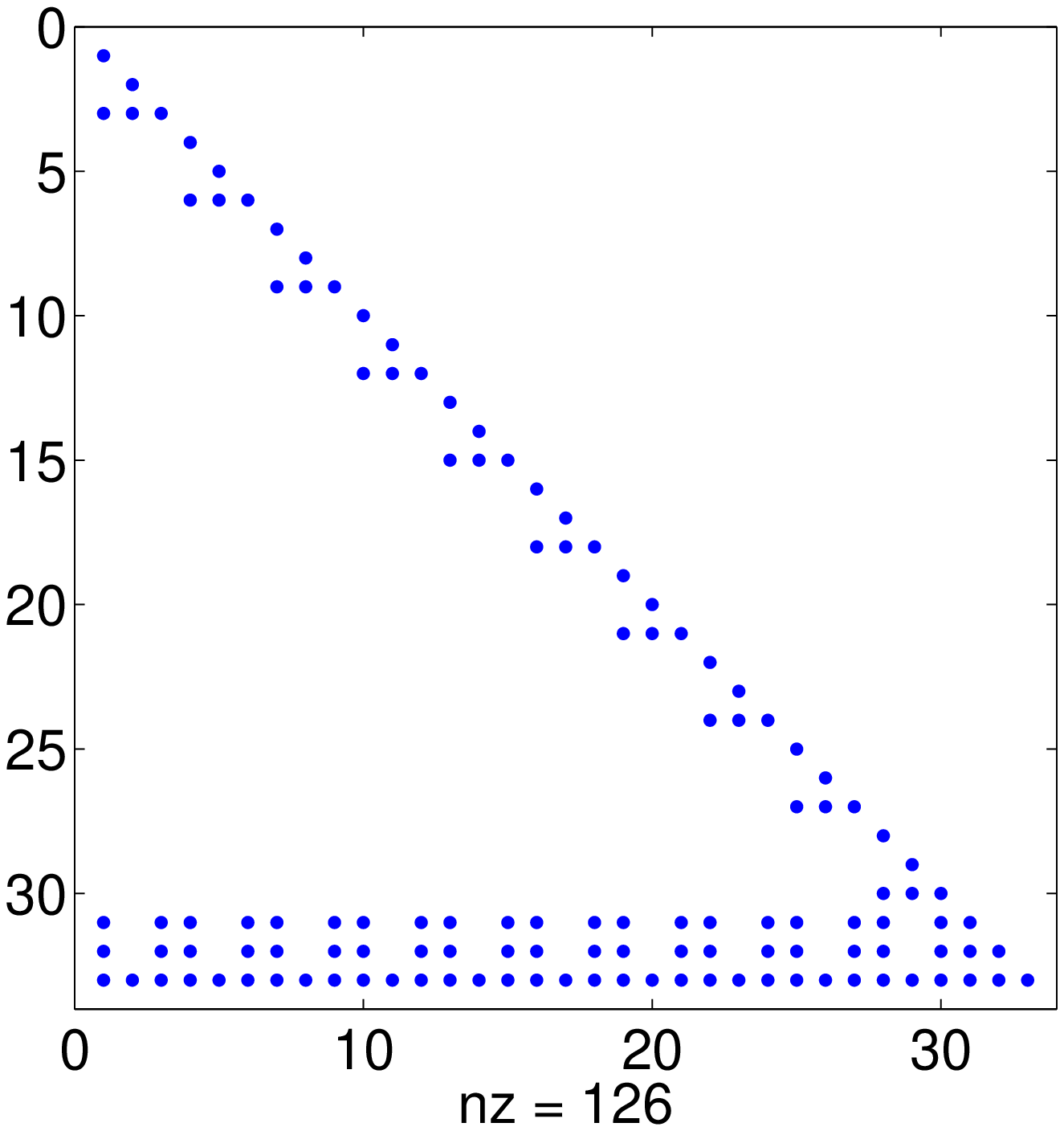}
}
\caption{Sparsity patterns of the permuted preconditioner and
the lower triangular factor $L$}
\label{fig5}
\end{figure}

\section{Conclusion}
\label{sec:concl}

We propose an efficient sparse preconditioner for the Continuation/GMRES method
for nonlinear MPC problems. The arithmetical cost of preconditioning is $O(N)$,
memory storage is $O(N)$, where $N$ is the number of gridpoints on the prediction horizon.

\vfill\pagebreak

\pagebreak

\end{document}